\documentclass[12pt]{amsart}       %was 12pt
\usepackage{txfonts}
\usepackage{algorithm}
\usepackage{algorithmic}
\usepackage{amssymb}
\usepackage{eucal}
\usepackage{graphicx}
\usepackage{amsmath}
\usepackage{amscd}
\usepackage[all]{xy}           %xypic macro for latex2.09
\usepackage{tikz}
\usepackage{tikz-qtree}
\usepackage{amsfonts,latexsym}
\usepackage{xspace}
\usepackage{epsfig}
\usepackage{float}
\usepackage{color}
\usepackage{fancybox}
\usepackage{colordvi}
\usepackage{multicol}
\usepackage{colordvi}
\usepackage{wasysym}
\usepackage[active]{srcltx} %SRC Specials for DVI Searching
\usepackage{CJK}
\ifpdf
  \usepackage[colorlinks,final,backref=page,hyperindex]{hyperref}
\else
  \usepackage[colorlinks,final,backref=page,hyperindex,hypertex]{hyperref}
\fi

\newcommand{\nc}{\newcommand}
\newcommand{\delete}[1]{}
\nc{\bfk}{{\bf k}}

\nc{\mlabel}[1]{\label{#1}}  % Use this to suppress names
\nc{\mcite}[1]{\cite{#1}}  % Use this to suppress names
\nc{\mref}[1]{\ref{#1}}  % Use this to suppress names
\nc{\mbibitem}[1]{\bibitem{#1}} % Use this to show number name
\delete{
\nc{\mlabel}[1]{\label{#1}  % Use the next two lines to show names
{\hfill \hspace{1cm}{\small\tt{{\ }\hfill(#1)}}}}
\nc{\mcite}[1]{\cite{#1}{\small{\tt{{\ }(#1)}}}}  % Use this lines to show names
\nc{\mref}[1]{\ref{#1}{{\tt{{\ }(#1)}}}}  % Use this lines to show names
\nc{\mbibitem}[1]{\bibitem[\bf #1]{#1}} % Use this to show name
}

\newtheorem{theorem}{Theorem}[section]
\newtheorem{proposition}[theorem]{Proposition}
\newtheorem{lemma}[theorem]{Lemma}
\newtheorem{corollary}[theorem]{Corollary}
\theoremstyle{definition}
\newtheorem{definition}[theorem]{Definition}
\newtheorem{notation}[theorem]{Notation}

\newtheorem{remark}[theorem]{Remark}

\newtheorem{tempex}[theorem]{Example}
\newtheorem{tempexs}[theorem]{Examples}
\newtheorem{temprmk}[theorem]{Remark}
\newtheorem{tempexer}{Exercise}[section]
\newenvironment{example}{\begin{tempex}\rm}{\end{tempex}}

\nc{\tred}[1]{\textcolor{red}{#1}} \nc{\tgreen}[1]{\textcolor{green}{#1}}
\nc{\tblue}[1]{\textcolor{blue}{#1}} \nc{\tpurple}[1]{\textcolor{purple}{#1}}

\nc{\hu}[1]{\tpurple{\underline{Hu:}#1 }}
\nc{\xing}[1]{\tblue{\underline{Xing:}#1 }}

\nc{\GS}{Gr\"obner-Shirshov\xspace}
\nc{\gsb}{Gr\"{o}bner-Shirshov basis\xspace}
\nc{\gsbs}{Gr\"{o}bner-Shirshov bases\xspace}
\nc\olie{operated Lie algebra\xspace}
\nc\olies{operated Lie algebras\xspace}
\nc\bfone{\mathbf{1}}

\nc\nas[1]{{#1}^\ast}
\nc\alsw[1]{{\rm ALSW}(#1)}
\nc\nlsw[1]{{\rm NLSW}(#1)}
\nc\clie[1]{[#1]}
\nc\lbar[1]{\overline{#1}}
\nc\suba[1]{|_{#1}}
\nc\coplie[1]{\bfk\nlsbw{#1}}
\nc\coplieo[2]{\bfk{\rm NLSBW}_{#2}(#1)}
\nc\lb[1]{\left[#1\right]}\nc\dt[1]{{t^{(#1)}}}
\nc{\Irr}{\mathrm{Irr}}
\nc\blw[1]{\lfloor#1\rfloor}
\nc\plie[1]{\mathfrak{S}(#1)}
\nc\plien[1]{\mathcal{N}(#1)}
\nc{\lc}{\lfloor} \nc{\rc}{\rfloor}
\nc\Id{\rm Id}\nc\sopm[1]{\mathfrak{S}^\star(#1)}

\nc\ordc{>_{{\rm Dl}}} \nc\ordqc{\geq_{{\rm Dl}}}
\nc\ord{>_{{\rm IM }}}\nc\ordq{\geq _{\rm IM}}
\nc\ordd{>_{{\rm IM}}}\nc\ordqd{\geq_{{\rm IM}}}
\nc\ordb{>_{{\rm IM}}}\nc\ordqb{\geq_{{\rm IM}}}
\nc\alsbw[1]{{\rm ALSBW}_{\ordq}(#1)} \nc\nlsbw[1]{{\rm NLSBW}_{\ordq}(#1)}
\nc\alsbwo[2]{{\rm ALSBW}_{#2}(#1)} \nc\nlsbwo[2]{{\rm NLSBW}_{#2}(#1)}\nc\bws[1]{{\lfloor#1\rfloor}}\nc\oplie{{\rm OLie}(X)}
\nc{\dep}{{\rm dep}}
\nc\ordt{\geq_{\rm dt}}
\nc\ordtl{>_{\rm dt}}

\begin{document}

\title[Some new operated Lie polynomial identities and Gr\"obner-Shirshov bases]{Some new operated Lie polynomial identities and Gr\"obner-Shirshov bases}

\author{Huhu Zhang} \address{School of Mathematics and Statistics,
Lanzhou University, Lanzhou, 730000, China}
\email{zhanghh20@lzu.edu.cn}

\author{Xing Gao$^{*}$}\thanks{*Corresponding author}
\address{School of Mathematics and Statistics, Lanzhou University
Lanzhou, 730000, China;
Gansu Provincial Research Center for Basic Disciplines of Mathematics
and Statistics
Lanzhou, 730070, China;
School of Mathematics and Statistics
Qinghai Nationalities University, Xining, 810007, China
}
\email{gaoxing@lzu.edu.cn}

\author{Tingzeng Wu}
\address{School of Mathematics and Statistics, Qinghai Nationalities University, Xining, Qinghai 810007,China;Qinghai Institute of Applied Mathematics, Xining, Qinghai 810007, China}
\email{mathtzwu@163.com}

\author{Xinyang Feng}
\address{College of Science, Northwest A$\&$F University, Yangling, Shaanxi, 712100, P.R. China}
\email{fxy1012@126.com}

\date{\today}

\begin{abstract}
Bremner and Elgendy developed a classification of operated polynomial identities for linear operators on associative algebras, encompassing both classical and newly discovered cases. Within the framework of Rota's Program, each of these new operated associative polynomial identities was shown to be Gr\"obner-Shirshov. This naturally led to a question posed by Guo and collaborators: is each corresponding operated Lie polynomial identity also Gr\"obner-Shirshov?
In this paper, we provide an affirmative answer by proving that each such Lie analogue indeed is Gr\"obner-Shirshov, thereby enriching the development of Rota's Program on algebraic operators within the Lie algebraic setting.
\end{abstract}

\makeatletter
\@namedef{subjclassname@2020}{\textup{2020} Mathematics Subject Classification}
\makeatother
\subjclass[2020]{
    18M70,  %Algebraic operads, cooperads, and Koszul duality
%    18M65, % Non-symmetric operads, multicategories, generalized multicategories
%	05C05,   %Trees
	05A05,   %words
%	05E99,   %Algebraic combinatorics/none of the above
	16W99, %Rings and algebras with additional structure/none of the above
%	12H05, %differential algebra
%	17B37, %Yang-Baxter equations and Rota-Baxter operators
%	16T30,  %Connections with combinatorics
%	16S10, %Rings determined by universal properties (free algebras, coproducts, adjunction of inverses, etc.)
%	13P10, %Grobner bases; other bases for ideals and modules
%	16T10, %Bialgebras
}

\keywords{Rota's Programm; \GS basis; Linear operator; Lie algebra}

\maketitle

\tableofcontents

\setcounter{section}{0}

\allowdisplaybreaks

\section{Introduction}
This paper addresses a question posed in~\cite{ZGG}: is each operated Lie polynomial identity (OLPI), corresponding to the new classes of operated polynomial identities (OPIs) on associative algebras introduced in~\cite{BE}, Gr\"obner-Shirshov? We provide an affirmative answer by proving that each such OLPI---both the newly introduced and classical ones (such as the OLPI for the averaging operator)---indeed is Gr\"obner-Shirshov, including cases that had previously remained unresolved. In particular, we identify new classes of good OLPIs---those that are Gr\"obner-Shirshov---whose associative counterparts in the work of Bremner and Elgendy exhibit the same property. These new OLPIs warrant further study, much like their associative analogues. Our results contribute to the broader development of Rota's Program on Algebraic Operators within the Lie algebraic setting, as envisioned in~\cite{ZGG}.

\subsection{Rota's Program on Algebraic Operators for associative algebras}
Linear operators characterized by the algebraic identities they satisfy have played a central role in both classical and modern mathematics, attracting growing attention in recent years. On associative algebras, many such operators arise naturally from diverse areas of mathematics and physics. Some notable examples include:
\begin{enumerate}
\item Homomorphism operators, originating in algebra and Galois theory, satisfy the identity
  $$f(xy)=f(x)f(y).$$

\item Derivations, fundamental in analysis and differential geometry~\cite{Ko, VS}, obey the Leibniz rule
  $$d(xy)=d(x)y+xd(y).$$

\item  Rota-Baxter operators, introduced by Baxter in the context of probability theory~\cite{Ba}, satisfy
$$P(x)P(y)=P(P(x)y+xP(y)+\lambda xy).$$

\item Modified Rota-Baxter operators (or Hilbert-type operators), arising in signal processing and fluid mechanics~\cite{Co, Tr}, satisfy
$$P(x)P(y) = P(xP(y)) + P(P(x)y) + \lambda xy.$$

\item Nijenhuis operators, introduced by  Cari\~nena et al.~\cite{CGM} to study quantum bi-Hamiltonian systems and constructed by
analogy with Poisson-Nijenhuis geometry, from the relative Rota-Baxter algebras~\cite{Uc}, satisfy
$$P(x)P(y) = P(xP(y) + P(x)y - P(xy)).$$

\item TD operators, first appearing in combinatorial contexts~\cite{Le}, satisfy
$$P(x)P(y) = P(xP(y) + P(x)y - xP(1)y).$$

\item Reynolds operators, with origins in turbulence theory~\cite{Re}, are defined by
$$P(x)P(y) = P(xP(y) + P(x)y- P(x)P(y)).$$

%\item Averaging operators, formally introduced by Kamp\'{e} de F\'{e}riet~\cite{KL}, satisfy
%$$P(x)P(y) = P(xP(y)) = P(P(x)y).$$
%These were implicitly studied earlier by Reynolds~\cite{Re} through idempotent Reynolds operators.
\end{enumerate}
These operators not only emerge from foundational mathematical theories but also find applications in diverse areas such as combinatorics, quantum groups, operads, Hopf algebras, integrable systems, and the renormalization of quantum field theories~\cite{Ag, Bai, BBGN, CK, Go, Go1, GK, Ni, Ro, ZGG20}. For further references and developments, see~\cite{GGZ, GGZ1,GSZ}.

Motivated by their ubiquity and structural importance, Gian-Carlo Rota proposed a systematic program to classify all possible algebraic identities that can be satisfied by linear operators on algebras. As Rota expressed:
\begin{quote}
{\em In a series of papers, I have tried to show that other linear operators satisfying
algebraic identities may be of equal importance in studying certain algebraic
phenomena, and I have posed the problem of finding all possible algebraic identities that can be satisfied by a linear operator on an algebra.}
\end{quote}
This broad and ambitious initiative, now known as Rota's Program on Algebraic Operators, goes beyond compiling known operator identities---it seeks to discover and classify all potential algebraic operator identities relevant to mathematical structures. Guo and collaborators~\mcite{GG} formalized this program in the context of associative algebras, using the tools of rewriting systems and Gr\"obner-Shirshov bases to analyze OPIs in a systematic and computationally tractable framework.

\subsection{Rota's Program on Algebraic Operators for Lie algebras}
Linear operators have also been extensively studied in the setting of Lie algebras, where they play roles analogous to those in associative algebras. For instance, derivations are central to the construction of Lie algebra complexes and their homologies. The Rota-Baxter operator on Lie algebras originally arose as an operator form of the classical Yang-Baxter equation and has since become a key tool in the theory of integrable systems~\mcite{BGN,RS1,RS2,STS}. For further developments, see~\mcite{GuV,GvK}. Motivated by these connections, it is natural to extend Rota's program on Algebraic Operators---initially formulated for associative algebras---to the realm of Lie algebras.

In recent work, Guo et al. have reformulated Rota's program for Lie algebras within the framework of rewriting systems and Gr\"{o}bner-Shirshov bases~\cite{ZGG}. Specifically, by interpreting operator identities as elements in the operated Lie polynomial algebra, one can analyze whether a given operator behaves well---i.e., is good---on Lie algebras through its associated rewriting system and Gr\"{o}bner-Shirshov basis.

Using this approach, Guo et al. have identified several classes of such good operators on Lie algebras, including modified Rota-Baxter operators, differential-type operators, and Rota-Baxter-type operators~\cite{ZGG}. Despite these advances, a complete resolution of Rota's program for algebraic operators---whether for associative or Lie algebras---remains a highly challenging task, even within the powerful frameworks of rewriting systems and Gr\"{o}bner-Shirshov bases.

\subsection{Motivation and outline of the paper.}
A different approach to Rota's Program on Algebraic Operators for associative algebras was explored in~\cite{BE}, where several new classes of OPIs were introduced. It was later shown in~\cite{WZG} that each of these OPIs is Gr\"obner-Shirshov in the associative setting. This naturally led to the question raised in~\cite{ZGG}: are the corresponding OLPIs also each Gr\"obner-Shirshov?

In this paper, we provide an affirmative answer to this question. As a result, we identify additional classes of good operators on Lie algebras in the sense of Rota's Program on Algebraic Operators.
It is worth emphasizing that verifying that OLPIs are Gr\"obner-Shirshov is more subtle than proving the Gr\"obner-Shirshov property of the corresponding associative identities. This complexity arises for two main reasons. First, the computations involve more terms---mirroring the fact that the Jacobi identity has one more term than the associative law. Second, the basis elements in free operated Lie algebras, namely the non-associative Lyndon-Shirshov bracketed words, are structurally more intricate than the bracketed words forming the basis for free operated associative algebras.

\vspace{0.3cm}

\noindent{\bf Outline of the paper.}
Section~\mref{se:gsb} reviews the construction of free operated Lie algebras using Lyndon-Shirshov bracketed words and recalls the Gr\"obner-Shirshov theory for operated Lie algebras. In Section~\mref{se:ol}, we formulate Lie algebraic analogues OLPIs of the OPIs from~\cite[Theorems 5.1 and 6.12]{BE}. We then prove that all OLPIs of operated degree 1 are Gr\"obner-Shirshov with respect to the invariant monomial order $\ordt$ (Theorem~\mref{tm:d1}).
Finally, we establish that all OLPIs of operated degree 2 are also Gr\"obner-Shirshov, with respect to the invariant monomial orders $\ordt$ or $\ordqc$ (Theorem~\mref{tm:sr}).

\vspace{0.3cm}

\noindent
{\bf Notation.}  Throughout this paper, let $\bfk$ be a field of characteristic zero, which serves as the base field for all vector spaces, algebras, and linear maps. For a set $X$, denote by $\bfk X$ the $\bfk$-vector space spanned by $X$, and let $\blw{X} := {\blw{x} \mid x \in X}$ be a disjoint copy of $X$. Let $S(X)$ (resp. $M(X)$) denote the free semigroup (resp. free monoid) on $X$, consisting of words (resp. words including the empty word $\bfone$) over the alphabet $X$. The free magma on $X$, denoted $\nas{X}$, consists of non-associative binary words built from $X$, recursively defined such that each element is either an element of $X$ or a bracketed pair $(w) = ((u)(v))$ for $(u), (v) \in \nas{X}$.

\section{\gsbs of free operated Lie algebras} \mlabel{se:gsb}
In this section, we review the theory of Gr\"obner-Shirshov bases for free operated Lie algebras~\cite{QC,ZGG}, as developed in terms of Lyndon-Shirshov bracketed words.

\subsection{Lyndon-Shirshov words}
Let $(X, \geq)$ be a well-ordered set. Define the lex-order $\geq_{\rm lex}$ on the free monoid $M(X)$ over $X$ by
\begin{enumerate}
	\item $\bfone>_{\rm lex} u$ for all nonempty word $u$, and
	\item for any $u = xu'$ and  $v = yv'$ with $x, y\in X$,
$$ u>_{\rm lex} v\,\text{ if }\,  x > y, \text{ or }\, x = y\,\text{ and }\, u'>_{\rm lex} v'.$$
\end{enumerate}
For example, let $x>y$. Then $\bfone>_{\rm lex}x>_{\rm lex}xx>_{\rm lex}xy>_{\rm lex}\cdots>_{\rm lex}y>_{\rm lex}yx>_{\rm lex}yy.$
\begin{definition} Let $(X, \geq)$ be a well-ordered set.
\begin{enumerate}
\item An associative word $w\in S(X)$ is called an {\bf associative Lyndon-Shirshov word} on $X$ with respect to the lex-order $\geq_{\rm lex}$, if $w =uv >_{\rm lex}vu$ for every decomposition $w = uv$ of $w$ with $u,v\in S(X)$.
\item A non-associative word $(w)\in \nas{X}$  is called a {\bf non-associative Lyndon-Shirshov word} on $X$ with respect to the lex-order $\geq_{\rm lex}$, provided
\begin{enumerate}
  \item the corresponding associative word $w$ is an associative Lyndon-Shirshov word on $X$;
  \item if $(w) = ((u)(v))$, then both $(u)$ and $(v)$ are non-associative Lyndon-Shirshov words on $X$;
  \item if $(w) = ((u)(v))$ and $(u) = ((u_1 )(u_2 ))$, then $v \geq_{\rm lex} u_2.$
\end{enumerate}
\end{enumerate}
\end{definition}

\begin{remark}
\begin{enumerate}
\item A word $w$ is an associative Lyndon-Shirshov word if and only if
for any decomposition of $w = uv$ with $u,v\in S(X)$, we have
$w>_{\rm lex} v$~\mcite{BC}.

\item Non-associative Lyndon(-Shirshov) words form a linear basis for a free Lie algebra~\cite{BC}.
\end{enumerate}
\end{remark}

Denote by $\alsw{X}$ (resp. $\nlsw{X}$) the set of all associative (resp. non-associative) Lyndon-Shirshov words on a well-ordered set $X$ with respect to the lex-order $\geq_{\rm lex}$. For any $w\in\alsw{X}$, there exists a unique
procedure, called the {\bf Shirshov standard bracketing}~\cite{BC}, to give a non-associative Lyndon-Shirshov word $[w]$. Furthermore, $\nlsw{X} =\left\{[w]\,|\,w \in\alsw{X}\right\}.$
%\begin{equation}
%\notag \nlsw{X} =\left\{[w]\,|\,w \in\alsw{X}\right\}.
%\label{eq:ssbw}
%\end{equation}

\begin{example}
Let $X=\{x,y,z\}$ with $x>y>z$. Then
\begin{enumerate}
  \item $xy$, $xyz,xzy\in \alsw{X}$ and $yx$, $yxz$, $yzx$, $zxy$, $zyx$ $\notin\alsw{X}$.
  \item $(xy)$, $(x(yz)), ((xz)y)\in \nlsw{X}$ and $(yx),((xy)z) \notin \nlsw{X}$.
  \item We display three examples of the Shirshov standard bracketing
  $$[xy]=(xy), [xyz]=(x(yz))\,\text{ and }\,[xzy]=((xz)y).$$
\end{enumerate}
\end{example}

\subsection{Lyndon-Shirshov bracketed words}
This subsection is devoted to collecting some definitions and results related to the Gr\"obner-Shirshov bases of operated Lie algebras~\mcite{QC,ZGG}.

\begin{definition}~\mcite{QC}
\begin{enumerate}
\item An {\bf \olie} is a Lie algebra $L$ together with a linear map
$P_L: L\to L$.
\item A {\bf morphism} from an \olie $(L_1,P_{L_1})$ to an \olie $(L_2,P_{L_2})$
is a Lie algebra homomorphism $f : L_1\to L_2$ such that
$f\circ P_{L_1}=P_{L_2}\circ f.$
\end{enumerate}
\end{definition}

Let $X$ be a set. The sets of associative and non-associative bracketed words on $X$ by direct systems $\plie{X}_n$ and $\plien{X}_n, n\geq 0$ defined recursively on $n$.
For the initial step of $n=0$, define $\plie{X}_0:=S(X)  \text{  and } \plien{X}_0:=\nas{X}.$ For the inductive step, define
$$\plie{X}_{n+1}:= S(X\cup \blw{\plie{X}_{n}})\,\text{ and }\, \plien{X}_{n+1}:= \nas{(X\cup \blw{\plien{X}_{n}})}.$$
Finally, define the direct limits
\vspace{-.2cm}
$$\plie{X}:=\bigcup_{n\geq0}\plie{X}_{n}\,\text{ and }\, \plien{X}:=\bigcup_{n\geq0}\plien{X}_{n}.$$
Elements of $\plie{X}$ (resp. $\plien{X}$) are called the {\bf associative  (resp. non-associative) bracketed words on $X$.}
We also denote an element of $\plie{X}$ (resp. $\plien{X}$) by $w$ (resp. $(w)$).

\begin{notation}
\begin{enumerate}
	\item Every element $w$ of $\plie{X}$ can be uniquely written in the form
$w=w_1\cdots w_k,$
	for $w_1,\ldots,w_k\in X\cup \blw{\plie{X}},k\geq 1$. We call $w_i$ {\bf prime} and $|w|:=k$ the {\bf breadth} of $w$.
	
	\item Define the {\bf depth} of $w\in\plie{X}$ to be ${\rm dep(w)} := {\rm min}\{n\,|\,w\in\plie{X}_n\}.$
	
	\item For any $(w)\in\plien{X}$, there exists a unique $w\in\plie{X}$ by forgetting the brackets of $(w)$. Then we can define the {\bf depth} of $(w)\in\plien{X}$ to be ${\rm dep((w))} := {\rm dep(w)}.$ This agrees with $\min \{n\,|\, (w)\in \plien{X}_n\}$.
	
	\item The {\bf degree} of $w\in\plie{X}$, denoted by ${\rm deg}(w)$, is defined to be the total number of occurrences of all $x\in X$ and $\blw{~}$ in $w$.
\item The {\bf operated degree} of $w\in\plie{X}$, denoted by ${\rm odeg}(w)$, is defined to be the total number of occurrences of all $\blw{~}$ in $w$.
\end{enumerate}
\end{notation}

\begin{example}
Let $u=\blw{xy\blw{z}y}xy\in\plie{X}$ and  $(v)=x((\blw{y}(xy))z)\in \plien{X}$  with $x,y,z\in X$. Then
$$|u|=3, \, {\rm dep}(u)=2,\, {\rm deg}(u)=8 \text{ and }\, {\rm odeg}(u)=2,$$
and
 $$v=x\blw{y}xyz\in \plie{X},{\rm dep}((v))={\rm dep}(v)=1\,\text{ and }\, {\rm odeg}((v))=1.$$
\end{example}

Let $\star$ be a symbol not in a set $X$. A word in $\plie{X\sqcup \star}$ is called a {\bf $\star$-word} on $X$ if $\star$ appears only once.
The set of all $\star$-words on $X$ is denoted by $\sopm{X}.$
\begin{definition}
Let $X$ be a set. A {\bf monomial order} on $\plie{X}$ is a well order $\geq$ on $\plie{X}$ such that
\vspace{-.1cm}
\begin{equation}
\quad  u >  v \Rightarrow  q\suba{u} >  q\suba{v},\,\text{ for all } u,v\in \plie{X} \text{ and } q\in\sopm{X}.
\mlabel{eq:mono} \notag
\end{equation}
\end{definition}

Let $\geq$ be a monomial order on $\plie{X}$.
For any $f\in \bfk\plie{X}$, let $\lbar{f}$ denote the {\bf leading monomial} of $f$.

\begin{definition}\mcite{ZGG}
  An order $\ordq$ on $\plie{X}$ is called {\bf invariant} if, for all prime elements $u_1, \ldots, u_n \in \plie{X}$ and $\sigma\in S_n$, we have
\vspace{-.1cm}
\begin{equation}
	u_1 \cdots u_n  \ordq u_{\sigma(1)} \cdots u_{\sigma(n)} \Longleftrightarrow u_1 \cdots u_n  \succeq_{\rm lex} u_{\sigma(1)} \cdots u_{\sigma(n)},
	\mlabel{eq:dlor}
\end{equation}
where $\succeq$ is the restriction of $\geq$ to the set of prime elements $X\sqcup\blw{\plie{X}}$.
\mlabel{de:inv}
\end{definition}

We are going to recall the construction of the free operated Lie algebras~\mcite{ZGG}. For a fixed invariant monomial order $\ordq$ on $\plie{X}$, define
$$\alsbw{X}_0:=\alsw{{X}}$$
and then
$$\nlsbw{X}_0:=\nlsw{{X}}=\left \{[w]\,\left |\,w\in \alsbw{X}_0\right .\right \}$$
with respect to the order $\succeq_{{\rm lex}}$.
For the recursion step, for any given $n\geq 0$, assume that we have defined
$$\alsbw{X}_{n}$$
and then
$$\nlsbw{X}_{n}=\left \{[w]\,\left|\,w\in \alsbw{X}_{n}\right . \right\}$$
with respect to the order $\succeq_{{\rm lex}}$.
Then we first define
$$\alsbw{X}_{n+1}:=\alsw{X\sqcup \blw{\alsbw{X}_{n}}}$$
with respect to the order $\succeq_{{\rm lex}}$.
We then denote
$$\big[X\sqcup \blw{\alsbw{X}_{n}}\big]:=\left\{[w]:=
\left\{\left .
\begin{array}{ll}
w, & \hbox{if  $w\in X$}, \\
\blw{[w']}, & \hbox{if $w=\blw{w'}$}
\end{array}
\right.\,\right|\, w \in X\sqcup \blw{\alsbw{X}_{n}}\right\}$$
and define
\begin{equation}
\nlsbw{X}_{n+1}:=\nlsw{\big[X\sqcup \blw{\alsbw{X}_{n}}\big]}=\left\{[w]\,\left|\,w\in\alsbw{X}_{n+1}\right.\right\}
\mlabel{eq:bnlsbw}
\end{equation}
with respect to the order $\succeq_{\rm lex}$.
We finally denote
$$\alsbw{X}:=\bigcup_{n\geq0}\alsbw{X}_n\,\text{ and }\,\nlsbw{X}:=\bigcup_{n\geq0}\nlsbw{X}_n.$$
We have
\begin{equation}
\nlsbw{X}=\left\{[w]\,\left|\,w\in\alsbw{X}\right .\right\}.
\mlabel{eq:astolie}
\end{equation}
Elements of $\alsbw{X}$ (resp. $\nlsbw{X}$) are called {\bf associative} (resp. {\bf non-associative}) {\bf Lyndon-Shirshov bracketed words} with respect to $\ordq$.

\subsection{\gsbs of free operated Lie algebras}
The $\nlsbw{X}$ is a linear basis of the free operated Lie algebra~\cite[Theorem~2.13]{ZGG}.
There are two kinds of nontrivial compositions.

\begin{lemma}\cite{ZGG}
Let $f\in \coplie{Z}$ be monic and $q\suba{\lbar{\clie{f}}}\in \alsbw{Z}$. Then
$$\clie{[q\suba{f}]_{\lbar{\clie{f}}}}=q\suba{\clie{f}}+\sum_i \alpha_i q_i\suba{\clie{f}},$$
where each $\alpha_i\in\bfk$, $q_i\in\sopm{Z}$ and $q\suba{\lbar{\clie{f}}}\ord q_i\suba{\lbar{\clie{f}}}$.
\mlabel{lem:jeq}
\end{lemma}

\begin{definition}
Let $X$ be a well-ordered set. Let $f, g \in\coplie{X}$ be monic.
\begin{enumerate}
\item  If there are $u, v, w\in \plie{X}$ such that $w = \lbar{\clie{f}}u = v\lbar{\clie{g}}$ with
$\max\{ |\lbar{\clie{f}}|, |\lbar{\clie{g}}|\}< |w| < |\lbar{\clie{f}}| + |\lbar{\clie{g}}|$, then
$$\langle f,g \rangle_w:=\langle f,g \rangle^{u,v}_w:= [fu]_{\lbar{\clie{f}}} - [vg]_{\lbar{\clie{g}}}$$
is called the {\bf intersection composition of $f$ and $g$ with respect to $w$}.
\label{item:intcompl}
\item  If there is $q\in \sopm{X}$ such that $w = \lbar{\clie{f}} = q\suba{\lbar{\clie{ g}}},$ then
$$\langle f,g \rangle_w:=\langle f,g \rangle^q_w := f - [q\suba{g}]_{\lbar{\clie{g}}}$$
is called the {\bf including composition of $f$ and $g$ with respect to $w$}.
\label{item:inccompl}
\end{enumerate}
\label{definition:compl}
\end{definition}
\begin{definition}\mcite{ZGG}
\label{def:gsbl}
Let $X$ be a well-ordered set. Let $S\subseteq\coplie{X}$ be monic.
\begin{enumerate}
\item An element $f\in\coplie{X}$ is called {\bf  trivial modulo $(S, w)$} if
$$f =\sum_i \alpha_i [q_i\suba{s_i}]_{\lbar{\clie{s_i}}}$$
with $w> q_i\suba{\lbar{\clie{s_i}}}$ for each special normal $s_i$-word $[q_i\suba{s_i}]_{\lbar{\clie{s_i}}}$, where each $ \alpha_i\in\bfk$, $q_i\in \sopm{X}$, $s_i\in S$.
In this case, we write $f\equiv 0 \mod(S, w).$
\item We call $S$ a {\bf \gsb} in $\coplie{Z}$ with respect to the  invariant monomial order $\ordq$ if, for all pairs $f, g \in S$, every intersection composition of the form $\langle f,g\rangle^{u,v}_w$ is trivial modulo $(S, w)$, and every including composition of the form $\langle f, g\rangle ^q_w$ is trivial modulo $(S, w)$.
\end{enumerate}
\end{definition}

The following \textit{Composition-Diamond lemma} is the cornerstone of Gr\"obner-Shirshov bases theory.

\begin{lemma}\cite{ZGG}\label{lem:cd}
Let $X$ be a well-ordered set, $\ordq$ an invariant monomial order on $\coplie{X}$ and ${S} \subseteq \coplie{X} $ monic. Denote $\Id({S})$ the ideal of $\coplie{X}$ generated by ${S}$. Then the following statements are equivalent.
\begin{enumerate}
    \item  ${S}$ is a Gr\"obner-Shirshov basis in $\coplie{X}.$
    \item  For all $f \neq 0$ in $\Id( S )$ , $\lbar{\clie{f}} = q\suba{\lbar{\clie{s}}}\in\alsbw{X}$ for some $q\in\sopm{X}$ and $s \in S$.
    \item $\coplie{X}=\bfk \Irr( S )\oplus \Id(S)$ and $\Irr(S)$ is a $\bfk$-basis of $\coplie{X}/ \Id(S)$, where
$$\Irr(S):=\left\{[w]\,\big|\, w\in\alsbw{X}, w\neq q\suba{\lbar{\clie{s}}}\,\text{ for }\, s\in S\,\text{ and }\, q\in\sopm{X}\right\}.$$
\end{enumerate}
\end{lemma}

\section{\gsbs for operated Lie polynomial identities}\mlabel{se:ol}
In this section, we prove that the Lie algebraic analogues of the operated polynomial identities in~\cite[Theorems 5.1 and 6.12]{BE} each form a Gr\"obner-Shirshov basis.

\begin{definition} Let $X$ be a well-ordered set and $\ordq$ an invariant monomial order on $\coplie{X}$.
\begin{enumerate}
  \item Elements of $\coplie{X}$ are called {\bf operated Lie polynomials}.
  \item For $\phi\in\coplie{X}$, we call $\phi=0$ (or simply $\phi$) an {\bf operated Lie polynomial identity (OLPI)}.
  \item Let $\Phi\subseteq\coplie{X}$ be a system of monic OLPIs.
We call $\Phi$ a {\bf \GS system} if $S_{\Phi}$ is a \gsb in $\coplie{X}$, where
$$S_{\Phi}=\{\phi(u_1,\ldots,u_k)\,|\,u_1,\ldots,u_k\in \coplie{X}, \phi\in \Phi\}\subseteq\coplie{X}.$$
\end{enumerate}
\end{definition}

The following are the Lie algebraic version of the operated polynomial identities in~\cite[Theorems 5.1 and 6.12]{BE}.
\begin{definition}\label{defn:nopi}
Let $X$ be a a set. For any $x,y\in X$, we define the following OLPIs.
\begin{enumerate}
\item \label{item:d1}{\bf OLPIs of operated degree 1}:
\begin{equation}\mlabel{ea:d1}
\begin{aligned}
&\blw{[xy]} , \,&&\blw{[xy]} - [x\blw{y}],  \,&&\blw{[xy]} - [\blw{x}y], \\
&\blw{[xy]} -[\blw{x}y] - [x\blw{y}],  \,&&[\blw{x}y] ,  \,&&[x\blw{y}] .
\end{aligned}
\end{equation}
\item\label{item:d2} {\bf OLPIs  of operated degree 2}:
\allowdisplaybreaks
\begin{equation}\mlabel{ea:d2}
\begin{aligned}
&\blw{\blw{[xy]}}, \,&&\blw{\blw{[xy]}} - [x\blw{\blw{y}}],  \,&&\blw{\blw{[xy]}} - [\blw{\blw{x}}y], \\
&\blw{\blw{[xy]}} -[\blw{\blw{x}}y] -[ x\blw{\blw{y}}],  \,&&[\blw{\blw{x}}y] ,  \,&&[x\blw{\blw{y} }],
\end{aligned}
\end{equation}
and
\allowdisplaybreaks
\begin{equation}\label{ea:d3}
\begin{aligned}
&[\blw{x}\blw{y}] - \blw{[\blw{x}y]} - \blw{[x \blw{y}]}&&\text{(Rota-Baxter)},\\
&[\blw{x}\blw{y}]- \blw{[\blw{x}y]} - \blw{[x \blw{y}]} + \blw{\blw{[x y]}}&&\text{(Nijenhuis)},\\
&[\blw{x}\blw{y}] - \blw{[x \blw{y}]} &&\text{(average)},\\
&[\blw{x}\blw{y}] - \blw{[\blw{x}y]} &&\text{(inverse average)},\\
&[x \blw{\blw{y}}]+\blw{\blw{[x y] }} + \blw{[x \blw{y }]} &&\text{(New identity A (right))},\\
&[\blw{\blw{x}}y]+\blw{\blw{[x y]} }+ \blw{[\blw{x}y]} &&\text{(New identity A (left))},\\
& \blw{\blw{[x y]} }+ \alpha\blw{[x \blw{y}]}- (\alpha + 1)[x \blw{\blw{y}}] \text{ with } 0\neq \alpha \in\bfk , &&\text{(New identity B (right))},\\
&\blw{\blw{[x y]} } + \beta\blw{[\blw{x}y]} - (\beta + 1)[\blw{\blw{x}}y] \text{ with } 0\neq \beta \in\bfk ,&&\text{(New identity B (left))},\\
&[\blw{\blw{x}}y]+\blw{\blw{[x y]}}  + [x \blw{\blw{y}}] + 2[\blw{x}\blw{y}] - 2\blw{[\blw{x}y]} - 2\blw{[x \blw{y}]}&&\text{(New identity C)},\\
&[\blw{\blw{x}}y]-\blw{[\blw{x}y]}  &&\text{(P 1 )},\\
&\blw{[\blw{x}y]} &&\text{(P 2 )},\\
&[x \blw{\blw{y}}]-\blw{[x \blw{y}]}  &&\text{(P 3 )},\\
&\blw{[x \blw{y}]} &&\text{(P 4 )},\\
&[\blw{x}\blw{y}] &&\text{(P 5 )}.
\end{aligned}
\end{equation}
\end{enumerate}
\end{definition}

We recall two examples of invariant monomial orders on $\plie{X}$ for later applications.

\noindent
{\bf First invariant monomial order $\ordqc$ on $\plie{X}$:} Let $(X, \geq)$ be a well-ordered set.
Take $u=u_1\cdots u_m$ and $v=v_1\cdots v_n$ in $\plie{X}$, where $u_i$ and $v_j$ are prime. Define
$u\ordc v $ inductively on $\dep(u)+\dep(v)\geq 0$.
For the initial step of $\dep(u)+\dep(v) = 0$, we have $u,v\in S(X)$ and define $u\ordc v$ by
$u >_{\rm deg-lex} v$, that is, $$u \ordc v \, \text{ if }\, ({\rm deg}(u), |u|,  u_1, \ldots, u_m) >({\rm deg}(v), |v|,  v_1, \ldots, v_n) \, \text{ lexicographically}.$$
Here notice that ${\rm deg}(u)= |u|$ and ${\rm deg}(v)= |v|$.
For the induction step, if $u = \lc u'\rc$ and $v = \lc v' \rc$, then define
$u\ordc v \,\text{ if }\, u' \ordc v'.$
Otherwise, define
\vspace{-.1cm}
$$u\ordc v \, \text{ if }\, ({\rm deg}(u), |u|,  u_1, \ldots, u_m) >({\rm deg}(v), |v|,  v_1, \ldots, v_n) \, \text{ lexicographically}.$$
Then $\ordqc$ is a monomial order~\cite{QC} and is invariant.

\noindent
{\bf Second invariant monomial order $\geq_{\rm dt}$ on $\plie{X}$:}
  Let $(X, \geq)$ be a well-ordered set.  Denote by $\deg_X(u)$ the
number of $x \in X$ in $u$ with repetition. Define the order $\geq_{\rm dt}$ on $\plie{X}$ as follows. For any $u=u_1\cdots u_m$ and $v=v_1\cdots v_n$, where $u_i$ and $v_j$ are prime. Define
$u>_{\rm dt} v $ inductively on $\dep(u)+\dep(v)\geq 0$.
For the initial step of $\dep(u)+\dep(v) = 0$, we have $u,v\in S(X)$ and define $u>_{\rm dt} v$ if $u>_{\rm deg-lex} v$, that is
$$(\deg_X(u), u_1, \ldots, u_m) > (\deg_X(v), v_1, \ldots, v_n) \, \text{ lexicographically}.$$
For the induction step, if $u = \lc u'\rc$ and $v = \lc v' \rc$, define $u>_{\rm dt} v$ if ${u'} >_{\rm dt} {v'}.$
If $u = \lc u'\rc$ and $v\in X$, define $u >_{\rm dt} v$.
Otherwise, define
$$u>_{\rm dt} v \, \text{ if }\, (\deg_X(u),  u_1, \ldots, u_m) >(\deg_X(v),  v_1, \ldots, v_n) \, \text{ lexicographically}.$$
Then $\ordt$ is a monomial order~\cite{GSZ}, which is also invariant.

\subsection{OLPIs of operated degree 1}\mlabel{se:d1}
In this subsection, we proved that each OLPI in Eq.~\eqref{ea:d1} is \GS with respect to the invariant monomial order $\ordt$.

\begin{remark}\mlabel{re:d1g}
Notice that there are three monomial OLPIs
$$\blw{[xy]}, \,[\blw{x}y] \,\text{ and }\, [x\blw{y}],$$
which are \GS with respect to the invariant monomial order $\ordt$, respectively.
\end{remark}

\begin{proposition}\mlabel{pr:d1g}
Let $X$ be a well-ordered set.
The OLPIs
\allowdisplaybreaks
\begin{align*}
\phi(x,y):=~&\blw{[xy]} -[\blw{x}y] - [x\blw{y}],\\
\psi(x,y):=~&\blw{[xy]} - [x\blw{y}],\\
\varphi(x,y):=~&\blw{[xy]} - [\blw{x}y]
\end{align*}
are respectively \GS with respect to the invariant monomial order $\ordt$.
\end{proposition}
\begin{proof}
The OLPI $\phi(x,y)$ is of differential type,  which is \GS with respect to the invariant monomial order $\ordt$~\cite[Theorem 4.8]{ZGG}.
We only prove that OLPI $\psi(x,y)$ is \GS, as the case of $\varphi(x,y)$ is similar.

Notice that
$$S_{\psi}=\{\psi(u,v)\,|\,u,v\in \alsbwo{X}{\ordt} \}=\{\blw{[uv]} - [u\blw{v}], [u\blw{u}]\,|\,u\ordtl v\in \alsbwo{X}{\ordt}\}$$
and
$\lbar{\psi(u,v)}=\blw{uv}$ for $u\ordtl v\in \alsbwo{Z}{\ordt}$.
Since the compositions  involving $[u\blw{u}]$ are trivial, it suffices to consider the compositions  involving $\blw{[uv]} - [u\blw{v}]$ with $u\ordtl v$.
These compositions are listed as below:
\allowdisplaybreaks
\begin{eqnarray*}
&&\langle \psi(uv,w), \psi(u,vw) \rangle_{w_1}\, \text{ for }\,  w_1=\blw{uvw}  \in \alsbwo{X}{\ordt},\\
&&\langle \psi(q\suba{\blw{uv}},w), \psi(u,v)) \rangle_{w_2}\,  \text{ for }\, w_2=\blw{q\suba{\blw{uv}}w} \in \alsbwo{X}{\ordt}, q\in\sopm{X},\\
&&\langle \psi(u,q\suba{\blw{vw}}), \psi(v,w))\rangle_{w_3}\,  \text{ for }\, w_3=\blw{uq\suba{\blw{vw}}} \in \alsbwo{X}{\ordt}, q\in\sopm{X}.
\end{eqnarray*}
Now we check that all the compositions are trivial. By direct computation, for the first one, we have
\allowdisplaybreaks
\begin{eqnarray*}
&&\langle \psi(uv,w), \psi(u,vw) \rangle_{w_1}\\
&=&\psi(uv,w)-\psi(u,vw)\\
&=&\blw{[(uv)w]} - [uv\blw{w}]-\blw{[u(vw)]}+[u\blw{vw}]\\
&=&\blw{[(uw)v]} - [uv\blw{w}]+[u\blw{vw}]\\
&\equiv&[uw\blw{v}] - [uv\blw{w}]+[u\blw{vw}]\\
&=&0 \mod(S_\psi, w_1).
\end{eqnarray*}
For the second one, by Lemma~\mref{lem:jeq}, we have
$$\blw{[q\suba{\psi(u,v)}w]}_{\blw{uv}}=\blw{[q\suba{\blw{uv}}w]} - \blw{[q\suba{u\blw{v}}w]}+\sum_i\alpha_i\big(\blw{[q_i\suba{\blw{uv}}w]} - \blw{[q_i\suba{u\blw{v}}w]}\big)$$
for some $\alpha_i\in\bfk$, $q_i\in\sopm{X}$ with $q\suba{\blw{uv}}\ordtl q_i\suba{\blw{uv}}$.
Hence
\allowdisplaybreaks
\begin{eqnarray*}
&&\langle \psi(q\suba{\blw{uv}},w), \psi(u,v)) \rangle_{w_2}\\
&=&\psi(q\suba{\blw{uv}},w)-\blw{[q\suba{\psi(u,v)}w]}_{\blw{uv}}\\
&=&\blw{[q\suba{\blw{uv}}w]} - [q\suba{\blw{uv}}\blw{w}]-\blw{[q\suba{\blw{uv}}w]} + \blw{[q\suba{u\blw{v}}w]}-\sum_i\alpha_i\big(\blw{[q_i\suba{\blw{uv}}w]} - \blw{[q_i\suba{u\blw{v}}w]}\big)\\
&=&- [q\suba{\blw{uv}}\blw{w}]+ \blw{[q\suba{u\blw{v}}w]}-\sum_i\alpha_i\big(\blw{[q_i\suba{\blw{uv}}w]} - \blw{[q_i\suba{u\blw{v}}w]}\big)\\
&\equiv&- [q\suba{u\blw{v}}\blw{w}]+ \blw{[q\suba{u\blw{v}}w]}-\sum_i\alpha_i\big(\blw{[q_i\suba{u\blw{v}}w]} - \blw{[q_i\suba{u\blw{v}}w]}\big)\\
&=&0 \mod(S_\psi, w_2).
\end{eqnarray*}
The third one is similar to the case of second one.
Therefore, $\psi(x,y)$ is \GS.
\end{proof}

Now we arrive at our first main result of this paper.
\begin{theorem}\mlabel{tm:d1}
Let $X$ be a well-ordered set. Six OLPIs in Eq.~\eqref{ea:d1}
are respectively \GS with respect to the invariant monomial order $\ordt$.
\end{theorem}
\begin{proof}
It follows from Remark~\mref{re:d1g} and Proposition~\mref{pr:d1g}.
\end{proof}

\subsection{OLPIs of operated degree 2}\mlabel{se:d2}
In this subsection, we proved that each OLPI in Eqs.~\eqref{ea:d2} and~\eqref{ea:d3} is \GS with respect to the invariant monomial orders $\ordt$ or $\ordqc$.
Parallel to Theorem~\mref{tm:d1}, we give the following result.

\begin{theorem}\mlabel{tm:d2}
Let $X$ be a well-ordered set. Six OLPIs in Eq.~\eqref{ea:d2}
are respectively \GS with respect to the invariant monomial order $\ordt$.
\end{theorem}

\begin{proof}
First, three monomial OLPIs in Eq.~\eqref{ea:d2} are \GS, respectively.
Next, other three OPIs in Eq.~\eqref{ea:d2} are respectively \GS,
by replacing $\blw{~}$ by $\blw{\blw{~}}$ in the proof of Proposition~\mref{pr:d1g}.
\end{proof}

We are in a position to prove that each OLPI in Eq.~\eqref{ea:d3} is \GS with respect to the invariant monomial orders $\ordt$ or $\ordqc$.
As a preparation, we expose the following result.

\begin{proposition}\mlabel{pr:rb}
Let $X$ be a well-ordered set. The OLPIs
\allowdisplaybreaks
  \begin{equation*}
\begin{aligned}
&[\blw{x}\blw{y}] - \blw{[\blw{x}y]} - \blw{[x \blw{y}]}&&\text{\rm (Rota-Baxter)},\\
&[\blw{x}\blw{y}]- \blw{[\blw{x}y]} - \blw{[x \blw{y}]} + \blw{\blw{[x y]}}&&\text{\rm(Nijenhuis)},\\
&[\blw{x}\blw{y}] - \blw{[x \blw{y}]} &&\text{\rm(average)},\\
&[\blw{x}\blw{y}] - \blw{[\blw{x}y]} &&\text{\rm(inverse average)}
\end{aligned}
\end{equation*}
are respectively \GS with respect to the invariant monomial order $\ordqc$.
\end{proposition}
\begin{proof}
The Rota-Baxter OLPI and Nijenhuis OLPI are respectively \GS with respect to the invariant monomial order $\ordqc$~\mcite{QC,ZGG}.
We only need to show that average OLPI is \GS,  as the case of inverse average OLPI is similar.

Denote $\phi(x,y) :=[\blw{x}\blw{y}] - \blw{[x \blw{y}]}$. Then
$$S_{\phi}=\{\phi(u,v)\,|\,u,v\in \alsbwo{X}{\ordqc} \}=\{[\blw{u}\blw{v}] - \blw{[u \blw{v}]}, \blw{[u \blw{u}]}\,|\,u\ordc v\in \alsbwo{X}{\ordqc}\}$$
and
$\lbar{\phi(u,v)}=\blw{u}\blw{v}$ for $u\ordc v\in \alsbwo{Z}{\ordqc}$.
Since the compositions  involving $\blw{[u \blw{u}]}$ are trivial, we are left to consider  the compositions  involving $\blw{[uv]} - [u\blw{v}]$ with $u\ordc v$.
These compositions are listed as below:
\begin{eqnarray*}
&&\langle\phi(u,v), \phi(v,w)\rangle_{w_1}\,\text{ for }\,  w_1=\blw{u}\blw{v}\blw{w}, u\ordc v \ordc w,\\
&&\langle\phi(q\suba{\blw{u}\blw{v}},w), \phi(u,v)\rangle_{w_2}\,\text{ for }\,  w_2=\blw{q\suba{\blw{u}\blw{v}}}\blw{w}, u\ordc v , q\suba{\blw{u}\blw{v}}\ordc w,\\
&&\langle\phi(u,q\suba{\blw{v}\blw{w}}), \phi(v,w)\rangle_{w_3}\,\text{ for }\,  w_3=\blw{u}\blw{q\suba{\blw{v}\blw{w}}}, u\ordc\blw{q\suba{\blw{u}\blw{v}}}, v\ordc w .
\end{eqnarray*}
The first composition is trivial from that
\allowdisplaybreaks
\begin{eqnarray*}
&&\langle\phi(u,v), \phi(v,w)\rangle_{w_1}\\
&=& [\phi(u,v)\blw{w}]\suba{\blw{u}\blw{v}}-[\blw{u}\phi(v,w)]\suba{\blw{v}\blw{w}}\\
&=& [[\blw{u}\blw{w}]\blw{v}] +[\blw{u}[\blw{v}\blw{w}] ]- [\blw{[u \blw{v}]}\blw{w}]-[\blw{u}[\blw{v}\blw{w}] ]+ [\blw{u}\blw{[v \blw{w}]}]\\
&=& [[\blw{u}\blw{w}]\blw{v}] - [\blw{[u \blw{v}]}\blw{w}]+ [\blw{u}\blw{[v \blw{w}]}]\\
&\equiv& [\blw{[u\blw{w}]}\blw{v}]-\blw{[[u\blw{v}]\blw{w}]}+\blw{[u\blw{[v\blw{w}]}]}\\
&\equiv& \blw{[[u\blw{w}]\blw{v}]}-\blw{[[u\blw{v}]\blw{w}]}+\blw{[u\blw{[v\blw{w}]}]}\\
&=& -\blw{[u[\blw{v}\blw{w}]]}+\blw{[u\blw{[v\blw{w}]}]}\\
&\equiv& -\blw{[u\blw{[v\blw{w}]}]}+\blw{[u\blw{[v\blw{w}]}]}\\
&=&0 \mod(S_\phi, w_1).
\end{eqnarray*}
For the second one, by Lemma~\mref{lem:jeq}, we have
$$[\blw{q\suba{\phi(u,v)}}\blw{w}]_{\blw{u}\blw{v}}=[\blw{q\suba{\blw{u}\blw{v}}}\blw{w}] -[\blw{q\suba{\blw{u\blw{v}}}}\blw{w}] +\sum_i\alpha_i\big([\blw{q_i\suba{\blw{u}\blw{v}}}\blw{w}] -[\blw{q_i\suba{\blw{u\blw{v}}}}\blw{w}] \big),$$
for some $\alpha_i\in\bfk$, $q_i\in\sopm{X}$ with $q\suba{\blw{u}\blw{v}}\ordc q_i\suba{\blw{u}\blw{v}}$, and so
\allowdisplaybreaks
\begin{eqnarray*}
&&\langle\phi(q\suba{\blw{u}\blw{v}},w), \phi(u,v)\rangle_{w_2}\\
&=&\phi(q\suba{\blw{u}\blw{v}},w)-[\blw{q\suba{\phi(u,v)}}\blw{w}]_{\blw{u}\blw{v}}\\
&=&[\blw{q\suba{\blw{u}\blw{v}}}\blw{w}] - \blw{[q\suba{\blw{u}\blw{v}} \blw{w}]}-[\blw{q\suba{\blw{u}\blw{v}}}\blw{w}] +[\blw{q\suba{\blw{u\blw{v}}}}\blw{w}]\\
&&-\sum_i\alpha_i\big([\blw{q_i\suba{\blw{u}\blw{v}}}\blw{w}] -[\blw{q_i\suba{\blw{u\blw{v}}}}\blw{w}] \big)\\
&=& - \blw{[q\suba{\blw{u}\blw{v}} \blw{w}]} +[\blw{q\suba{\blw{u\blw{v}}}}\blw{w}]-\sum_i\alpha_i\big([\blw{q_i\suba{\blw{u}\blw{v}}}\blw{w}] -[\blw{q_i\suba{\blw{u\blw{v}}}}\blw{w}] \big)\\
&=& - [\blw{q\suba{\blw{u\blw{v}}}}\blw{w}] +[\blw{q\suba{\blw{u\blw{v}}}}\blw{w}]-\sum_i\alpha_i\big([\blw{q_i\suba{\blw{u\blw{v}}}}\blw{w}] -[\blw{q_i\suba{\blw{u\blw{v}}}}\blw{w}] \big)\\
&=&0 \mod(S_\phi,w_2).
\end{eqnarray*}
Similarly, the third composition is also trivial module $(S_\phi, w_3)$. Therefore, the average OLPI $\phi(x,y)$ is \GS. This completes the proof.
\end{proof}

\begin{remark}
With the invariant monomial order $\ordt$, the A, B and C type OLPIs  are not \GS, respectively. Since they have same leading monomial $\blw{\blw{uv}}$, there exist a composition of the form
$\langle \phi(uv,w), \phi(u, vw)\rangle_{\blw{\blw{uvw}}}$ with $uvw\in\alsbwo{X}{\ordt}$. However, such composition is not trivial module $(S_\phi, \blw{\blw{uvw}})$.
\end{remark}
We now check that the OLPIs associated to the new identities $A$  classified in~\mcite{BE} are  respectively \GS.
\begin{proposition}\mlabel{pr:a}
Let $X$ be a well-ordered set. The OLPIs
\allowdisplaybreaks
\begin{equation*}
  \begin{aligned}
    &[x \blw{\blw{y}}]+\blw{\blw{[x y] }} + \blw{[x \blw{y }]}  &&\text{\rm(New identity A (right))},\\
&[\blw{\blw{x}}y]+\blw{\blw{[x y]} }+ \blw{[\blw{x}y]}  &&\text{\rm(New identity A (left))}
  \end{aligned}
\end{equation*}
are respectively \GS with respect to the invariant monomial order $\ordqc$.
\end{proposition}

\begin{proof}
It suffices  to prove that  A (right) type OLPI is \GS,  as the proof of the left case is similar.
Define $\phi(x,y) :=[x \blw{\blw{y}}]+\blw{\blw{[x y] }} + \blw{[x \blw{y }]}$ and
$$S_\phi :=\{\phi(u,v)\,|\,u,v\in\alsbwo{X}{\ordqc} \}.$$
Then $\lbar{\phi(u,v)}=u\blw{\blw{v}}$ and there are two compositions
\begin{eqnarray*}
&&\langle\phi(q\suba{u \blw{\blw{v}}},w), \phi(u,v)\rangle_{w_1} \,\text{ for }\, w_1=q\suba{u \blw{\blw{v}}}\blw{\blw{w}}\in \alsbwo{X}{\ordqc},\\
&&\langle\phi(u, q\suba{v \blw{\blw{w}}}),\phi(v,w)\rangle_{w_2} \,\text{ for }\, w_2=u\blw{\blw{q\suba{v \blw{\blw{w}}}}}\in \alsbwo{X}{\ordqc}.
\end{eqnarray*}
For the first composition, by Lemma~\mref{lem:jeq}, we get
\begin{eqnarray*}
&&[q\suba{\phi(u,v)}\blw{\blw{w}}]_{u\blw{\blw{v}}}\\
&=&[q\suba{u \blw{\blw{v}}} \blw{\blw{w}}]+[q\suba{\blw{\blw{[uv] }}} \blw{\blw{w}}]+ [q\suba{\blw{[u \blw{v }]}} \blw{\blw{w}}]\\
&&+\sum_i\alpha_i\left([q_i\suba{u \blw{\blw{v}}} \blw{\blw{w}}]+[q_i\suba{\blw{\blw{[uv] }}} \blw{\blw{w}}]+ [q_i\suba{\blw{[u \blw{v }]}} \blw{\blw{w}}]\right),
\end{eqnarray*}
for some $\alpha_i\in\bfk$, $q_i\in\sopm{X}$ with $q\suba{u \blw{\blw{v}}}\ordc q_i\suba{u \blw{\blw{v}}}$, and hence
{\small
\begin{eqnarray*}
&&\langle\phi(q\suba{u \blw{\blw{v}}},w), \phi(u,v)\rangle_{w_1}\\
&=&\phi(q\suba{u \blw{\blw{v}}},w)-[q\suba{\phi(u,v)}\blw{\blw{w}}]_{u\blw{\blw{v}}}\\
&=&[q\suba{u \blw{\blw{v}}} \blw{\blw{w}}]+\blw{\blw{[q\suba{u \blw{\blw{v}}} w] }} + \blw{[q\suba{u \blw{\blw{v}}} \blw{w }]}
-[q\suba{u \blw{\blw{v}}} \blw{\blw{w}}]-[q\suba{\blw{\blw{[uv] }}} \blw{\blw{w}}]- [q\suba{\blw{[u \blw{v }]}} \blw{\blw{w}}]\\
&&-\sum_i\alpha_i\left([q_i\suba{u \blw{\blw{v}}} \blw{\blw{w}}]+[q_i\suba{\blw{\blw{[uv] }}} \blw{\blw{w}}]+ [q_i\suba{\blw{[u \blw{v }]}} \blw{\blw{w}}]\right)\\
&=&\blw{\blw{[q\suba{u \blw{\blw{v}}} w] }} + \blw{[q\suba{u \blw{\blw{v}}} \blw{w }]}
-[q\suba{\blw{\blw{[uv] }}} \blw{\blw{w}}]- [q\suba{\blw{[u \blw{v }]}} \blw{\blw{w}}]\\
&&-\sum_i\alpha_i\left([q_i\suba{u \blw{\blw{v}}} \blw{\blw{w}}]+[q_i\suba{\blw{\blw{[uv] }}} \blw{\blw{w}}]+ [q_i\suba{\blw{[u \blw{v }]}} \blw{\blw{w}}]\right)\\
&\equiv&
-\blw{\blw{[q\suba{\blw{[u \blw{v}]}} w] }} -\blw{\blw{[q\suba{\blw{ \blw{[uv]}}} w] }}-\blw{[q\suba{\blw{[u \blw{v}]}} \blw{w }]}-\blw{[q\suba{\blw{\blw{[u v]}}} \blw{w }]}
+[\blw{q\suba{\blw{\blw{[uv] }}} \blw{w}}]+[\blw{\blw{q\suba{\blw{\blw{[uv] }}} w}}]\\
&&+[\blw{q\suba{\blw{[u \blw{v }]}} \blw{w}}]+[\blw{\blw{q\suba{\blw{[u \blw{v }]}} w}}]
-\sum_i\alpha_i\Big(-[q_i\suba{\blw{[u \blw{v}]}} \blw{\blw{w}}]-[q_i\suba{\blw{ \blw{[uv]}}} \blw{\blw{w}}]\\
&&+[q_i\suba{\blw{\blw{[uv] }}} \blw{\blw{w}}]+ [q_i\suba{\blw{[u \blw{v }]}} \blw{\blw{w}}]\Big)\\
&=&0 \mod(S_\phi, w_1).
\end{eqnarray*}}
Similarly, the second  composition is also trivial $(S_\phi, w_2)$. Therefore, the OLPI $\phi(x,y)$ is \GS.
\end{proof}

We turn to consider the  B type OLPIs.

\begin{remark}\mlabel{re:rb}
Notice that if the monomial $\blw{\blw{x y} }$ is leading monomial of the following B type OLPIs
\begin{equation*}
  \begin{aligned}
& \blw{\blw{[x y]} }+ \alpha\blw{[x \blw{y}]}- (\alpha + 1)[x \blw{\blw{y}}] \text{ with } 0\neq \alpha \in\bfk , &&\text{(New identity B (right))},\\
&\blw{\blw{[x y]} } + \beta\blw{[\blw{x}y]} - (\beta + 1)[\blw{\blw{x}}y] \text{ with } 0\neq \beta \in\bfk ,&&\text{(New identity B (left))},
  \end{aligned}
\end{equation*}
then they are not \GS,  respectively. Since the base field $\bfk$ is of characteristic 0, we rewrite B type OLPIs as
\par
\noindent {\bf Case 1:} for $\alpha\neq -1$ and $\beta \neq -1$,
\begin{align}
& [x \blw{\blw{y}}]-\frac{1}{\alpha + 1}\blw{\blw{[x y]} }- \frac{\alpha}{\alpha + 1}\blw{[x \blw{y}]} \text{ with } 0\neq \alpha \in\bfk ,
\mlabel{eq:br1}\\
&[\blw{\blw{x}}y]-\frac{1}{\beta + 1}\blw{\blw{[x y]} } - \frac{\beta}{\beta + 1} \blw{[\blw{x}y]}  \text{ with } 0\neq \beta \in\bfk.\mlabel{eq:bl1}
\end{align}
\par
\noindent {\bf Case 2:} for $\alpha= -1$ and $\beta =-1$,
\begin{align}
& [\blw{x \blw{y}}]-\blw{\blw{[x y]} } ,\mlabel{eq:br2}\\
&[\blw{\blw{x}y}]-\blw{\blw{[x y]} }.\mlabel{eq:bl2}
\end{align}
\end{remark}

\begin{proposition}\mlabel{pr:bb}
Let $X$ be a well-ordered set. Four OLPIs in Eqs.~\eqref{eq:br1}, ~\eqref{eq:bl1}, ~\eqref{eq:br2} and~\eqref{eq:bl2}
are respectively \GS with respect to the invariant monomial order $\ordqc$.
\end{proposition}

\begin{proof}
We only consider the cases of Eqs.~\eqref{eq:br1} and~\eqref{eq:br2}, as others are similar.
Define
\begin{align*}
\phi(x,y):=~&[x \blw{\blw{y}}]-\frac{1}{\alpha + 1}\blw{\blw{[x y]} }- \frac{\alpha}{\alpha + 1}\blw{[x \blw{y}]},\\
\psi(x,y):=~&[\blw{x \blw{y}}]-\blw{\blw{[x y]} },
\end{align*}
and denote
\begin{align*}
S_\phi:=~&\{\phi(u,v)\,|\,u,v\in\alsbwo{X}{\ordqc} \},\\
S_\psi:=~&\{\psi(u,v)\,|\,u,v\in\alsbwo{X}{\ordqc} \}.
\end{align*}

We first check that $S_\phi$ is a \gsb. With the order $\ordqc$, $u \blw{\blw{v}}$ is the leading monomial of $\phi(u,v)$.
There are two compositions
\begin{eqnarray*}
&&\langle\phi(q\suba{u \blw{\blw{v}}},w), \phi(u,v)\rangle_{w_1} \,\text{ for }\, w_1=q\suba{u \blw{\blw{v}}}\blw{\blw{w}}\in \alsbwo{X}{\ordqc},\\
&&\langle\phi(u, q\suba{v \blw{\blw{w}}}),\phi(v,w)\rangle_{w_2} \,\text{ for }\, w_2=u\blw{\blw{q\suba{v \blw{\blw{w}}}}}\in \alsbwo{X}{\ordqc}.
\end{eqnarray*}
We only check that the first one is trivial, as the second one is similar. By Lemma~\mref{lem:jeq}, we obtain
\begin{eqnarray*}
[q\suba{\phi(u,v)}\blw{\blw{w}}]_{u\blw{\blw{v}}}
&=&[q\suba{u \blw{\blw{v}}} \blw{\blw{w}}]-\frac{1}{\alpha + 1}[q\suba{\blw{\blw{[uv]} }} \blw{\blw{w}}]- \frac{\alpha}{\alpha + 1} [q\suba{\blw{[u \blw{v}]}} \blw{\blw{w}}]\\
&&+\sum_i\alpha_i\left([q_i\suba{\phi(u,v)} \blw{\blw{w}}]\right),
\end{eqnarray*}
for some $\alpha_i\in\bfk$, $q_i\in\sopm{X}$ with $q\suba{u \blw{\blw{v}}}\ordc q_i\suba{u \blw{\blw{v}}}$. Hence
\allowdisplaybreaks{
\begin{eqnarray*}
&&\langle\phi(q\suba{u \blw{\blw{v}}},w), \phi(u,v)\rangle_{w_1}\\
&=&\phi(q\suba{u \blw{\blw{v}}},w)-[q\suba{\phi(u,v)}\blw{\blw{w}}]_{u\blw{\blw{v}}}\\
&=&[q\suba{u \blw{\blw{v}}} \blw{\blw{w}}]-\frac{1}{\alpha + 1}\blw{\blw{[q\suba{u \blw{\blw{v}}}w]} }- \frac{\alpha}{\alpha + 1}\blw{[q\suba{u \blw{\blw{v}}} \blw{w}]}\\
&&-[q\suba{u \blw{\blw{v}}} \blw{\blw{w}}]+\frac{1}{\alpha + 1}[q\suba{\blw{\blw{[uv]} }} \blw{\blw{w}}]+ \frac{\alpha}{\alpha + 1} [q\suba{\blw{[u \blw{v}]}} \blw{\blw{w}}]-\sum_i\alpha_i\left([q_i\suba{\phi(u,v)} \blw{\blw{w}}]\right)\\
&=&-\frac{1}{\alpha + 1}\blw{\blw{[q\suba{u \blw{\blw{v}}}w]} }- \frac{\alpha}{\alpha + 1}\blw{[q\suba{u \blw{\blw{v}}} \blw{w}]}
+\frac{1}{\alpha + 1}[q\suba{\blw{\blw{[uv]} }} \blw{\blw{w}}]+ \frac{\alpha}{\alpha + 1} [q\suba{\blw{[u \blw{v}]}} \blw{\blw{w}}]\\
&&-\sum_i\alpha_i\left([q_i\suba{\phi(u,v)} \blw{\blw{w}}]\right)\\
&\equiv&-\frac{1}{\alpha + 1}\blw{\blw{[q\suba{u \blw{\blw{v}}}w]} }- \frac{\alpha}{\alpha + 1}\blw{[q\suba{u \blw{\blw{v}}} \blw{w}]}
+\frac{1}{\alpha + 1}[q\suba{\blw{\blw{[uv]} }} \blw{\blw{w}}]+ \frac{\alpha}{\alpha + 1} [q\suba{\blw{[u \blw{v}]}} \blw{\blw{w}}]\\
&\equiv&-\frac{1}{(\alpha + 1)^2}\blw{\blw{[q\suba{\blw{\blw{[uv]} }}w]} }- \frac{\alpha}{(\alpha + 1)^2}\blw{\blw{[q\suba{\blw{[u \blw{v}]}}w]} }
-\frac{\alpha}{(\alpha + 1)^2}{\blw{[q\suba{\blw{\blw{[uv]} }}\blw{w}]} }
\\&&- \frac{\alpha^2}{(\alpha + 1)^2}{\blw{[q\suba{\blw{[u \blw{v}]}}\blw{w}]} }
+\frac{1}{(\alpha + 1)^2}[\blw{\blw{q\suba{\blw{\blw{[uv]} }} w}}]+\frac{\alpha}{(\alpha + 1)^2}[\blw{q\suba{\blw{\blw{[uv]} }} \blw{w}}]\\
&&+\frac{\alpha}{(\alpha + 1)^2}[\blw{\blw{q\suba{\blw{[u \blw{v}]}} w}}]+\frac{\alpha^2}{(\alpha + 1)^2}[\blw{q\suba{\blw{[u \blw{v}]}} \blw{w}}]\\
&=&0 \mod( S_\phi, w_1),
\end{eqnarray*}}
whence $S_\phi$ is a \gsb.

Finally we show that $S_\psi$ is a \gsb. With the order $\ordqc$, we have $\lbar{\psi(u,v)}=\blw{u \blw{v}}$.
There are two compositions
\begin{eqnarray*}
&&\langle\psi(q\suba{\blw{u \blw{v}}},w), \psi(u,v)\rangle_{w_3} \,\text{ for }\, w_3=\blw{q\suba{\blw{u \blw{v}}}\blw{w}}\in \alsbwo{X}{\ordqc},\\
&&\langle\psi(u, q\suba{\blw{v \blw{w}}}),\psi(v,w)\rangle_{w_4} \,\text{ for }\, w_4=\blw{u\blw{q\suba{\blw{v \blw{w}}}}}\in \alsbwo{X}{\ordqc}.
\end{eqnarray*}
For the first composition, by Lemma~\mref{lem:jeq},  we have
$$[\blw{q\suba{\psi(u,v)} \blw{w}}]_{\blw{u \blw{v}}}
=[\blw{q\suba{\blw{[u \blw{v}]}} \blw{w}}]-[\blw{q\suba{ \blw{\blw{[uv]}}} \blw{w}}]
+\sum_i\alpha_i[\blw{q_i\suba{\psi(u,v)} \blw{w}}],$$
for some $\alpha_i\in\bfk$, $q_i\in\sopm{X}$ with $q\suba{\blw{u \blw{v}}}\ordc q_i\suba{\blw{u \blw{v}}}$. So
\begin{eqnarray*}
&&\langle\psi(q\suba{\blw{u \blw{v}}},w), \psi(u,v)\rangle_{w_3}\\
&=&[\blw{q\suba{\blw{u \blw{v}}} \blw{w}}]-\blw{\blw{[q\suba{\blw{u \blw{v}}}w]} }-[\blw{q\suba{\psi(u,v)} \blw{w}}]_{\blw{u \blw{v}}}\\
&=&[\blw{q\suba{\blw{u \blw{v}}} \blw{w}}]-\blw{\blw{[q\suba{\blw{u \blw{v}}}w]} }
-[\blw{q\suba{\blw{[u \blw{v}]}} \blw{w}}]+[\blw{q\suba{ \blw{\blw{[uv]}}} \blw{w}}]
-\sum_i\alpha_i[\blw{q_i\suba{\psi(u,v)} \blw{w}}]\\
&=&-\blw{\blw{[q\suba{\blw{u \blw{v}}}w]} }+[\blw{q\suba{ \blw{\blw{[uv]}}} \blw{w}}]
-\sum_i\alpha_i[\blw{q_i\suba{\psi(u,v)} \blw{w}}]\\
&\equiv&-\blw{\blw{[q\suba{\blw{u \blw{v}}}w]} }+[\blw{q\suba{ \blw{\blw{[uv]}}} \blw{w}}]\\
&\equiv&-\blw{\blw{[q\suba{\blw{\blw{[u v]}}}w]} }+\blw{\blw{[q\suba{ \blw{\blw{[uv]}}} w]}}\\
&=&0 \mod(S_\psi, w_3).
\end{eqnarray*}
%[\blw{x \blw{y}}]-\blw{\blw{[x y]} }
Similarly, we obtain the second composition is trivial. Therefore $S_\psi$ is also a \gsb.
\end{proof}

As an immediate consequence of Remark~\mref{re:rb} and Proposition~\mref{pr:b}, we derive the following result.

\begin{corollary}\mlabel{pr:b}
Let $X$ be a well-ordered set. Two B type OLPIs in Eqs.~\eqref{ea:d3}
are respectively \GS with respect to the invariant monomial order $\ordqc$.
\end{corollary}

The following result focuses on the C type OLPI.

\begin{proposition}\mlabel{pr:c}
Let $X$ be a well-ordered set. The OLPI
\allowdisplaybreaks
\begin{equation*}
[\blw{\blw{x}}y] +\blw{\blw{[x y]}} +  [x \blw{\blw{y}}] + 2[\blw{x}\blw{y}] - 2\blw{[\blw{x}y]} - 2\blw{[x \blw{y}]} \quad\text{\rm(New identity C)}
\end{equation*}
is \GS with respect to the invariant monomial order $\ordqc$.
\end{proposition}
\begin{proof}
Denote
\begin{equation*}
\phi(x,y) :=[\blw{\blw{x}}y] +\blw{\blw{[x y]}} +  [x \blw{\blw{y}}] + 2[\blw{x}\blw{y}] - 2\blw{[\blw{x}y]} - 2\blw{[x \blw{y}]},
\end{equation*}
and
$$S_\phi :=\{\phi(u,v)\,|\,u,v\in\alsbwo{X}{\ordqc} \}.$$
With the order $\ordqc$, we have $\lbar{\phi(u,v)}=\blw{\blw{u}}v$. There are two compositions
\begin{eqnarray*}
&&\langle\phi(q\suba{ \blw{\blw{u}}v},w), \phi(u,v)\rangle_{w_1} \,\text{ for }\, w_1=\blw{\blw{q\suba{ \blw{\blw{u}}v}}}w\in \alsbwo{X}{\ordqc},\\
&&\langle\phi(u, q\suba{ \blw{\blw{v}}w}),\phi(v,w)\rangle_{w_2} \,\text{ for }\, w_2=\blw{\blw{u}}q\suba{ \blw{\blw{v}}w}\in \alsbwo{X}{\ordqc}.
\end{eqnarray*}
We only check the first one is trivial, as the second one is similar. Using Lemma~\mref{lem:jeq}, we obtain
\begin{eqnarray*}
&&[\blw{\blw{q\suba{ \phi(u,v)}}}w]_{\blw{\blw{u}}v}\\
&=& [\blw{\blw{q\suba{ \phi(u,v)}}}w]+\sum_i\alpha_i\Big([\blw{\blw{q_i\suba{ \phi(u,v)}}}w]\big)\\
&=&  [\blw{\blw{q\suba{   [\blw{\blw{u}}v]          }}}w]+[\blw{\blw{q\suba{   \blw{\blw{[uv]}}          }}}w]+[\blw{\blw{q\suba{     [u \blw{\blw{v}}]        }}}w]
+2[\blw{\blw{q\suba{  [\blw{u}\blw{v}]           }}}w]\\
&&-2[\blw{\blw{q\suba{      \blw{[\blw{u}v]}       }}}w]-2[\blw{\blw{q\suba{   \blw{[u \blw{v}]}          }}}w]
+\sum_i\alpha_i\Big([\blw{\blw{q_i\suba{ \phi(u,v)}}}w]\big),
\end{eqnarray*}
for some $\alpha_i\in\bfk$, $q_i\in\sopm{X}$ with $q\suba{u \blw{\blw{v}}}\ordc q_i\suba{u \blw{\blw{v}}}$.
Consequently,
\nc\phicna[2]{\blw{\blw{[#1#2]}} +  [#1 \blw{\blw{#2}}]+ 2[\blw{#1}\blw{#2}] - 2\blw{[\blw{#1}#2]}\\&& - 2\blw{[#1 \blw{#2}]}}
\nc\phicnb[2]{\blw{\blw{[#1#2]}} +  [#1 \blw{\blw{#2}}]+ 2[\blw{#1}\blw{#2}]\\&& - 2\blw{[\blw{#1}#2]} - 2\blw{[#1 \blw{#2}]}}
\nc\phicnc[2]{2\blw{\blw{[#1#2]}} +  2[#1 \blw{\blw{#2}}]\\&&+ 4[\blw{#1}\blw{#2}] - 4\blw{[\blw{#1}#2]} - 4\blw{[#1 \blw{#2}]}}
\nc\phica[2]{-2\blw{\blw{[#1#2]}}\\&& - 2 [#1 \blw{\blw{#2}}]- 4[\blw{#1}\blw{#2}] + 4\blw{[\blw{#1}#2]} + 4\blw{[#1 \blw{#2}]}\\}
\nc\phicb[2]{&&-2\blw{\blw{[#1#2]}} - 2 [#1 \blw{\blw{#2}}]- 4[\blw{#1}\blw{#2}] + 4\blw{[\blw{#1}#2]}\\&& + 4\blw{[#1 \blw{#2}]}}
\allowdisplaybreaks{\small
\begin{eqnarray*}
&&\langle\phi(q\suba{ \blw{\blw{u}}v},w), \phi(u,v)\rangle_{w_1}\\
&=& \phi(q\suba{ \blw{\blw{u}}v},w)-[\blw{\blw{q\suba{ \phi(u,v)}}}w]_{\blw{\blw{u}}v}\\
&=&[\blw{\blw{q\suba{ \blw{\blw{u}}v}}}w] +\blw{\blw{[q\suba{ \blw{\blw{u}}v}w]}} +  [q\suba{ \blw{\blw{u}}v} \blw{\blw{w}}] + 2[\blw{q\suba{ \blw{\blw{u}}v}}\blw{w}] - 2\blw{[\blw{q\suba{ \blw{\blw{u}}v}}w]} - 2\blw{[q\suba{ \blw{\blw{u}}v} \blw{w}]}\\
&&- [\blw{\blw{q\suba{   [\blw{\blw{u}}v]          }}}w]-[\blw{\blw{q\suba{   \blw{\blw{[uv]}}          }}}w]-[\blw{\blw{q\suba{     [u \blw{\blw{v}}]        }}}w]
-2[\blw{\blw{q\suba{  [\blw{u}\blw{v}]           }}}w]\\
&&+2[\blw{\blw{q\suba{      \blw{[\blw{u}v]}       }}}w]+2[\blw{\blw{q\suba{   \blw{[u \blw{v}]}          }}}w]
-\sum_i\alpha_i\Big([\blw{\blw{q_i\suba{ \phi(u,v)}}}w]\big)\\
&=&\blw{\blw{[q\suba{ \blw{\blw{u}}v}w]}} +  [q\suba{ \blw{\blw{u}}v} \blw{\blw{w}}] + 2[\blw{q\suba{ \blw{\blw{u}}v}}\blw{w}] - 2\blw{[\blw{q\suba{ \blw{\blw{u}}v}}w]} - 2\blw{[q\suba{ \blw{\blw{u}}v} \blw{w}]}\\
&&-[\blw{\blw{q\suba{   \blw{\blw{[uv]}}          }}}w]-[\blw{\blw{q\suba{     [u \blw{\blw{v}}]        }}}w]-2[\blw{\blw{q\suba{  [\blw{u}\blw{v}]           }}}w]
+2[\blw{\blw{q\suba{      \blw{[\blw{u}v]}       }}}w]+2[\blw{\blw{q\suba{   \blw{[u \blw{v}]}          }}}w]\\
&&-\sum_i\alpha_i\Big([\blw{\blw{q_i\suba{ \phi(u,v)}}}w]\big)\\
&\equiv&\blw{\blw{[q\suba{ \blw{\blw{u}}v}w]}} +  [q\suba{ \blw{\blw{u}}v} \blw{\blw{w}}] + 2[\blw{q\suba{ \blw{\blw{u}}v}}\blw{w}] - 2\blw{[\blw{q\suba{ \blw{\blw{u}}v}}w]} - 2\blw{[q\suba{ \blw{\blw{u}}v} \blw{w}]}\\
&&-[\blw{\blw{q\suba{   \blw{\blw{[uv]}}          }}}w]-[\blw{\blw{q\suba{     [u \blw{\blw{v}}]        }}}w]-2[\blw{\blw{q\suba{  [\blw{u}\blw{v}]           }}}w]
+2[\blw{\blw{q\suba{      \blw{[\blw{u}v]}       }}}w]+2[\blw{\blw{q\suba{   \blw{[u \blw{v}]}          }}}w]\\
&\equiv&
%%%%%%%%1:
\blw{\blw{[q\suba{ -\blw{\blw{[uv]}} -  [u \blw{\blw{v}}] - 2[\blw{u}\blw{v}] + 2\blw{[\blw{u}v]} + 2\blw{[u \blw{v}]}}w]}}
+  [q\suba{ -\blw{\blw{[uv]}} -  [u \blw{\blw{v}}] - 2[\blw{u}\blw{v}] + 2\blw{[\blw{u}v]} + 2\blw{[u \blw{v}]}} \blw{\blw{w}}]\\
&&+2[\blw{q\suba{-\blw{\blw{[uv]}} -  [u \blw{\blw{v}}] - 2[\blw{u}\blw{v}] + 2\blw{[\blw{u}v]} + 2\blw{[u \blw{v}]}}}\blw{w}]
-2\blw{[\blw{q\suba{ -\blw{\blw{[uv]}} -  [u \blw{\blw{v}}] - 2[\blw{u}\blw{v}] + 2\blw{[\blw{u}v]} + 2\blw{[u \blw{v}]}}}w]} \\
&&-2\blw{[q\suba{ -\blw{\blw{[uv]}} -  [u \blw{\blw{v}}] - 2[\blw{u}\blw{v}] + 2\blw{[\blw{u}v]} + 2\blw{[u \blw{v}]}} \blw{w}]}\\
%%%%%%%%%%2:
%-\blw{\blw{[uv]}} -  [u \blw{\blw{v}}] - 2[\blw{u}\blw{v}] + 2\blw{[\blw{u}v]} + 2\blw{[u \blw{v}]}
&&+\phicna{\blw{\blw{q\suba{\blw{\blw{[uv]}}}}}}{w}
+\phicnb{\blw{\blw{q\suba{     [u \blw{\blw{v}}]        }}}}{w}
\phicnc{\blw{\blw{q\suba{  [\blw{u}\blw{v}]           }}}}{w}
\phica{\blw{\blw{q\suba{      \blw{[\blw{u}v]}       }}}}{w}
\phicb{\blw{\blw{q\suba{   \blw{[u \blw{v}]}          }}}}{w}\\
&=&0 \mod(S_\phi, w_1),
\end{eqnarray*}}
as required.
\end{proof}

Now we are in a position to consider the P type OLPIs.
\begin{proposition}\mlabel{pr:p}
Let $X$ be a well-ordered set. The OLPIs
\allowdisplaybreaks
\begin{equation*}
  \begin{aligned}
&[\blw{\blw{x}}y] -\blw{[\blw{x}y]}  &&\text{\rm(P 1 )},\\
&\blw{[\blw{x}y]} &&\text{\rm(P 2 )},\\
&[x \blw{\blw{y}}]-\blw{[x \blw{y}]} &&\text{\rm(P 3 )},\\
&\blw{[x \blw{y}]} &&\text{\rm(P 4 )},\\
&[\blw{x}\blw{y}] &&\text{\rm(P 5 )}
  \end{aligned}
\end{equation*}
are respectively  \GS with respect to the invariant monomial order $\ordqc$.
\end{proposition}

\begin{proof}
The monomial OLPIs of P 2, P 4 and P 5 are \GS, respectively.
We are left to prove that the OLPI of P 1 is \GS, as the case of P 3 is similar. Define
$\phi(x,y):=[\blw{\blw{x}}y] -\blw{[\blw{x}y]}$
and
$$S_\phi:=\{\phi(u,v)\,|\,u,v\in\alsbwo{X}{\ordqc} \}.$$
Notice that $\lbar{\phi(u,v)}=\blw{\blw{u}}v$ and there are two compositions
\begin{eqnarray*}
&&\langle\phi(q\suba{ \blw{\blw{u}}v},w), \phi(u,v)\rangle_{w_1} \,\text{ for }\, w_1=\blw{\blw{q\suba{ \blw{\blw{u}}v}}}w\in \alsbwo{X}{\ordqc},\\
&&\langle\phi(u, q\suba{ \blw{\blw{v}}w}),\phi(v,w)\rangle_{w_2} \,\text{ for }\, w_2=\blw{\blw{u}}q\suba{ \blw{\blw{v}}w}\in \alsbwo{X}{\ordqc}.
\end{eqnarray*}
For the first one, using Lemma~\mref{lem:jeq}, we have
\begin{eqnarray*}
&&[\blw{\blw{q\suba{ \phi(u,v)}}}w]_{\blw{\blw{u}}v}\\
&=& [\blw{\blw{q\suba{ \phi(u,v)}}}w]+\sum_i\alpha_i\Big([\blw{\blw{q_i\suba{ \phi(u,v)}}}w]\big)\\
&=&  [\blw{\blw{q\suba{   [\blw{\blw{u}}v]          }}}w]
-[\blw{\blw{q\suba{  \blw{[\blw{u}v] }          }}}w]
+\sum_i\alpha_i\Big([\blw{\blw{q_i\suba{ \phi(u,v)}}}w]\big),
\end{eqnarray*}
for some $\alpha_i\in\bfk$, $q_i\in\sopm{X}$ with $q\suba{u \blw{\blw{v}}}\ordc q_i\suba{u \blw{\blw{v}}}$.
Therefore,
\begin{eqnarray*}
&&\langle\phi(q\suba{ \blw{\blw{u}}v},w), \phi(u,v)\rangle_{w_1}\\
&=& \phi(q\suba{ \blw{\blw{u}}v},w)-[\blw{\blw{q\suba{ \phi(u,v)}}}w]_{\blw{\blw{u}}v}\\
&=& [\blw{\blw{q\suba{ [\blw{\blw{u}}v]}}}w] -\blw{[\blw{q\suba{ [\blw{\blw{u}}v]}}w]}
-[\blw{\blw{q\suba{   [\blw{\blw{u}}v]          }}}w]
+[\blw{\blw{q\suba{  \blw{[\blw{u}v] }          }}}w]
-\sum_i\alpha_i\Big([\blw{\blw{q_i\suba{ \phi(u,v)}}}w]\big)\\
&=& -\blw{[\blw{q\suba{ [\blw{\blw{u}}v]}}w]}
+[\blw{\blw{q\suba{  \blw{[\blw{u}v] }          }}}w]
-\sum_i\alpha_i\Big([\blw{\blw{q_i\suba{ \phi(u,v)}}}w]\big)\\
&\equiv&
\blw{[\blw{q\suba{ \blw{[\blw{u}v]}}}w]}
-\blw{[\blw{q\suba{  \blw{[\blw{u}v] }          }}w]}
-\sum_i\alpha_i\Big([\blw{\blw{q_i\suba{ \phi(u,v)}}}w]\big)\\
&\equiv&-\sum_i\alpha_i\Big([\blw{\blw{q_i\suba{ \phi(u,v)}}}w]\big)\\
&\equiv&0 \mod( S_\phi, w_1).
\end{eqnarray*}
This completes the proof.
\end{proof}

The following theorem is the second main result of this paper.
\begin{theorem}\mlabel{tm:sr}
Let $X$ be a well-ordered set. Twenty OLPIs in Eqs.~\eqref{ea:d2} and~\eqref{ea:d3}
are respectively \GS with respect to the invariant monomial orders $\ordt$ or $\ordqc$.
\end{theorem}
\begin{proof}
  It follows from Theorem~\mref{tm:d2}, Corollary~\mref{pr:b} and Propositions~\mref{pr:rb}, ~\mref{pr:a}, ~\mref{pr:c} and~\mref{pr:p}.
\end{proof}

\smallskip

\noindent
{\bf Acknowledgments:} X. Gao is supported by the Natural Science Foundation of Gansu Province (25JRRA644), Innovative Fundamental Research Group Project of Gansu Province (23JRRA684) and Longyuan Young Talents of Gansu Province.
H. H. Zhang is supported by the Scientific Research Foundation of High-Level Talents of Yulin University (2025GK12) and Young Talent Fund of Association for Science and Technology in Shaanxi, China (20250530).
X. Y. Feng is supported by the Natural Science Project of Shaanxi Province (2022JQ-035).

\medskip
\noindent
{\bf Declaration of interests:} The authors have no conflicts of interest to disclose.

\noindent
{\bf Data availability:} Data sharing is not applicable to this article as no new data were created or
analyzed in this study.

\end{document}